\documentclass[12pt]{amsart}

\usepackage{amssymb}
\usepackage{bm}
\usepackage{graphicx}
\usepackage{psfrag}
\usepackage[usenames,dvipsnames]{xcolor}
\usepackage{enumerate}
\usepackage{url}

\usepackage{algpseudocode}

\usepackage{mathtools}

\usepackage{xy}
\input xy
\xyoption{all}

\usepackage{tikz}
\usetikzlibrary{positioning} 
\usetikzlibrary{graphs,graphs.standard}

\newcommand{\excise}[1]{}

\newtheorem{thm}{Theorem}[section]
\newtheorem{lemma}[thm]{Lemma}

\newtheorem{cor}[thm]{Corollary}
\newtheorem{prop}[thm]{Proposition}

\newtheorem{question}[thm]{Question}

\theoremstyle{definition}

\newtheorem{example}[thm]{Example}
\newtheorem{remark}[thm]{Remark}
\newtheorem{defn}[thm]{Definition}

\numberwithin{equation}{section}

\newcommand{\ring}[1]{\ensuremath{\mathbb{#1}}}

\newcommand\RR{\ring{R}}

\newcommand\ZZ{\ring{Z}}

\def\ol#1{{\overline {#1}}}

\begin{document}

\mbox{}
\title{Sequentially embeddable graphs}

\author[J.~Autry]{Jackson Autry}
\address{Mathematics and Statistics Department\\San Diego State University\\San Diego, CA 92182}
\email{jautry@sdsu.edu}

\author[C.~O'Neill]{Christopher O'Neill}
\address{Mathematics and Statistics Department\\San Diego State University\\San Diego, CA 92182}
\email{cdoneill@sdsu.edu}

\date{\today}

\begin{abstract}
We call a (not necessarily planar) embedding of a graph $G$ in the plane \emph{sequential} if its vertices lie in $\ZZ^2$ and the line segments between adjacent vertices contain no interior integer points.  
In this note, we prove (i) a graph $G$ has a sequential embedding if and only if $G$ is 4-colorable, and (ii) if $G$ is planar, then $G$ has a sequential planar embedding.  
\end{abstract}

\maketitle


\section{Introduction}
\label{sec:intro}

A planar embedding of a graph $G$ with $n$ vertices is called a \emph{straight-line} embedding (or an embedding \emph{on the grid}) if each vertex of $G$ lies in $\ZZ^2$ and each edge of $G$ is a line segment.  Every planar graph has such an embedding in the $n \times n$ grid that can be computed in linear time \cite{embedgrid}.  Several refinements of this result have been found, e.g.\ any 3-connected planar graph and its dual can be simultaneously embedded with straight line embeddings in a $(2n - 2) \times (2n - 2)$ grid \cite{embeddualgrid}.  Moreover, there is a recent trend in studying when two planar graphs with the same vertex set admit \emph{geometric RAC simultaneous drawings}, that is, simultaneous straight-line embeddings in which edges from different graphs intersect at right angles \cite{simultaneousplanar,simultaneousfewbends}.  

In this paper, we are interested in embeddings of the following form.  

\begin{defn}\label{d:sequential}
A (not necessarily planar) embedding of a graph $G$ in $\RR^d$ is \emph{sequential} if (i) the vertices of $G$ lie in $\ZZ^d$ and (ii) the line segments between adjacent vertices in~$G$ contain no interior integer points.  
\end{defn}

The following question was posed in \cite{rsegment} as a special case of a more general geometric family of hypergraphs arising from discrete geometry (see Section~\ref{sec:rsegment}).  We note that in~\cite{rsegment} it was shown some non-planar graphs have sequential embeddings (e.g.\ any bipartite graph), but not all graphs admit sequential embeddings (e.g.\ the complete graph $K_5$).  

\begin{question}[{\cite[Problem~2.7]{rsegment}}]\label{q:2segmentrealizable}
Which graphs admit a sequential embedding in $\RR^2$?
\end{question}

We provide a complete answer to Problem~\ref{q:2segmentrealizable} by proving that a graph $G$ admits a sequential embedding in $\RR^d$ if and only if $G$ is $2^d$-colorable (Corollary~\ref{c:2segmentrealizable}).  Additionally, we prove that if $G$ is planar, then $G$ admits a sequential planar embedding in $\RR^2$ (that is, a sequential embedding in which no edges cross) computable in linear time from any existing planar embedding (Theorem~\ref{t:planarseq}).  Lastly, in Section~\ref{sec:rsegment} we demonstrate that a characterization in terms of chromatic number is not possible for the more general family of hypergraphs considered in \cite{rsegment}.  

\section{A characterization of sequentially embeddable graphs}
\label{sec:2segment}

The main result of this section is Corollary~\ref{c:2segmentrealizable}, which characterizes which graphs admit sequential embeddings in $\RR^d$.  The bulk of the argument is contained in the proof of Theorem~\ref{t:2segmentrealizable}, which answers Question~\ref{q:2segmentrealizable}.  Our results use the following property of sequential line segments.  

\begin{lemma}\label{l:consecutive}
The line segment between $a, b \in \ZZ^d$ contains no interior integer points if~and only if $\gcd(a_1 - b_1, \ldots, a_d - b_d) = 1$.
\end{lemma}





\begin{thm}\label{t:2segmentrealizable}
If a graph is $4$-colorable, then it admits a sequential embedding in $\RR^2$.
\end{thm}

\begin{proof}
Fix a graph $G$ with a proper coloring by $\{1, 2, 3, 4\}$, and let 
$$\begin{array}{r@{}c@{}l@{\qquad}r@{}c@{}l}
C_1 &{}={}& \{(0, 6n) : n \in \ZZ\},
& 
C_2 &{}={}& \{(1, 2n) : n \in \ZZ\},
\\
C_3 &{}={}& \{(2, 1 + 2n) : n \in \ZZ\},
&
C_4 &{}={}& \{(3, 1 + 6n) : n \in \ZZ\}.
\end{array}$$
Consider any (not necessarily planar) embedding of $G$ in $\RR^2$ in which each vertex of color $i$ is sent to a distinct element of $C_i$.  To prove this embedding is sequential, it suffices to show any line segment from a point in $C_i$ to a point in $C_j$ is sequential.  

There are clearly no other points on a line segment from $C_i$ to $C_{i+1}$ for each $i \le 3$, as the $x$-coordinates of any points therein differ by exactly 1.  Moreover, any line segment from $C_1$ to $C_3$ or from $C_2$ to $C_4$ has vector coordinates of the form $(2,1 + 2n)$ for some $n \in \ZZ$, whose coordinates are clearly relatively prime.  
Lastly, any line segment from $C_1$ to $C_4$ has vector coordinates of the form $(3,1 + 6n)$ for some $n \in \ZZ$.  In every case, the proof is complete by Lemma~\ref{l:consecutive}.
\end{proof}

\begin{cor}\label{c:2segmentrealizable}
A graph $G$ has a sequential embedding in $\RR^d$ if and only if $G$ is $2^d$-colorable.
\end{cor}

\begin{proof}
First, suppose $G$ has a sequential embedding, and color the vertices of $G$ using the projection map 
$$\begin{array}{rcl}
\ZZ^d &\longrightarrow& (\ZZ/2\ZZ)^d \\
a &\longmapsto& (\ol a_1, \ldots, \ol a_d).
\end{array}$$
This yields a proper $2^d$-coloring of $G$ since endpoints of edges in $G$ must have distinct coordinate parity combinations by Lemma~\ref{l:consecutive}.  

For the converse direction, we will use the construction in the proof of Theorem~\ref{t:2segmentrealizable}.  Specifically, define $C_1, \ldots, C_4$ as in the proof of Theorem~\ref{t:2segmentrealizable}, and consider the sets
$$
\{(c,v) : c \in C_i\} \quad \text{for each} \quad v \in \{0,1\}^{d-2} \quad \text{and} \quad i = 1, 2, 3, 4.
$$
We see any line segment between points in distinct sets above contains no interior integer points by Lemma~\ref{l:consecutive} (indeed, if the endpoints differ past the second coordinate, they differ by 1 in that coordinate, and if the endpoints are identical past the second coordinate, then use Theorem~\ref{t:2segmentrealizable}), which completes the proof.  
\end{proof}

\section{Planar embeddings of 2-segment hypergraphs}
\label{sec:planar}

In this section, we show that every planar graph has a sequential planar embedding.  Our proof uses Theorem~\ref{t:straightline} (see~\cite{embedgrid}), the 4-color theorem, and a construction similar to the proof of Theorem~\ref{t:2segmentrealizable}.  

\begin{example}\label{e:planarnotplanar}
For some planar graphs, the proof of Theorem~\ref{t:2segmentrealizable} will never yield a planar embedding, regardless of the 4-coloring chosen.  Indeed, consider the graph~$G$ depicted in Figure~\ref{f:3segment}, obtained from $K_6$ (the complete graph on 6 vertices) by removing 3 disjoint pairs of non-adjacent vertices $\{a_1, a_2\}$, $\{b_1, b_2\}$, and $\{c_1, c_2\}$.  In any proper 4-coloring of $G$, at least two of these pairs (say, $\{a_1, a_2\}$ and $\{b_1, b_2\}$) must be monochromatic, so under the construction in Theorem~\ref{t:2segmentrealizable}, $a_1$ and $a_2$ must have the same $x$-coordinate, as must $b_1$ and $b_2$.  However, the resulting embedding cannot be planar with the given coordinate constraints, since the induced subgraph of $G$ on the vertex set $\{a_1, a_2, b_1, b_2\}$ is the complete bipartite graph $K_{2,2}$.  
\end{example}

\begin{thm}[{\cite[Theorem~1.2]{embedgrid}}]\label{t:straightline}
Every planar graph has a straight-line embedding.  
\end{thm}

\begin{thm}\label{t:planarseq}
Every planar graph has a sequential planar embedding.  
\end{thm}

\begin{proof}
Fix a planar graph $G$.  By Theorem~\ref{t:straightline}, $G$ has a planar embedding in which vertices lie in the integer lattice $\ZZ^2$, and edges are straight line segments.  Since $G$ has only finitely many vertices, there is some $\epsilon > 0$ so that each vertex can be independently moved within an $\epsilon \times \epsilon$ square without straight lines between adjacent vertices intersecting.  The planar embedding of $G$ can be freely scaled horizontally and (independently) vertically while preserving planarity.  First, dilate the embedding horizontally by some sufficiently large integer factor so the rectangular region around each vertex has width at least 12, and then translate the embedding as necessary to ensure the rectangular region around each vertex lies in the 1st quadrant.  Let $D = M!$, where $M$ is some integer greater than the $x$-coordinate of every integer point in the union of the rectangular regions around the vertices of $G$, and stretch the embedding of $G$ vertically so that the rectangular region around each vertex has height at least $D$.  Let
$$\begin{array}{r@{}c@{}l@{\qquad}r@{}c@{}l}
C_1 &{}={}& \{(12a ,Db) : a, b \in \ZZ\},
& 
C_2 &{}={}& \{(12a + 2, Db + 1) : a, b \in \ZZ\},
\\
C_3 &{}={}& \{(12a + 5, Db + 2): a, b \in \ZZ\},
&
C_4 &{}={}& \{(12a + 7, Db + 3): a, b \in \ZZ\},
\end{array}$$
so that the rectangular region around each vertex contains at least one point from each of the above sets.  By the four color theorem \cite{4color1,4color2}, $G$ can be properly colored by $\{1, 2, 3, 4\}$, and subsequently we can perturb each color-$i$ vertex $v$ so that it lies in~$C_i$.  
As in the proof of Theorem~\ref{t:2segmentrealizable}, we will show that any line segment between a point in $C_i$ and a point in $C_j$ with $i \ne j$ has no interior integer points.

The vector between any two points $(12a', Db') \in C_1$ and $(12a + 5, Db + 2) \in C_3$ is 
$$(12a + 5, Db + 2) - (12a', Db') = (12(a-a') + 5, D(b-b') + 2).$$
By construction, $(12(a-a') + 5) \mid D$, so 
$$\gcd(12(a-a') + 5, D(b-b') + 2) = \gcd(12(a-a') + 5, 2) = \gcd(5,2) = 1,$$
thereby proving any line segment from a point in $C_1$ to a point in $C_3$ is sequential by Lemma~\ref{l:consecutive}.  
By a similar argument, any line segment between a point in $C_i$ and a point in $C_j$ for $i \ne j$ contains no interior integer points, thereby completing the proof.  
\end{proof}

\begin{figure}
\begin{center}
\includegraphics[height=1.3in]{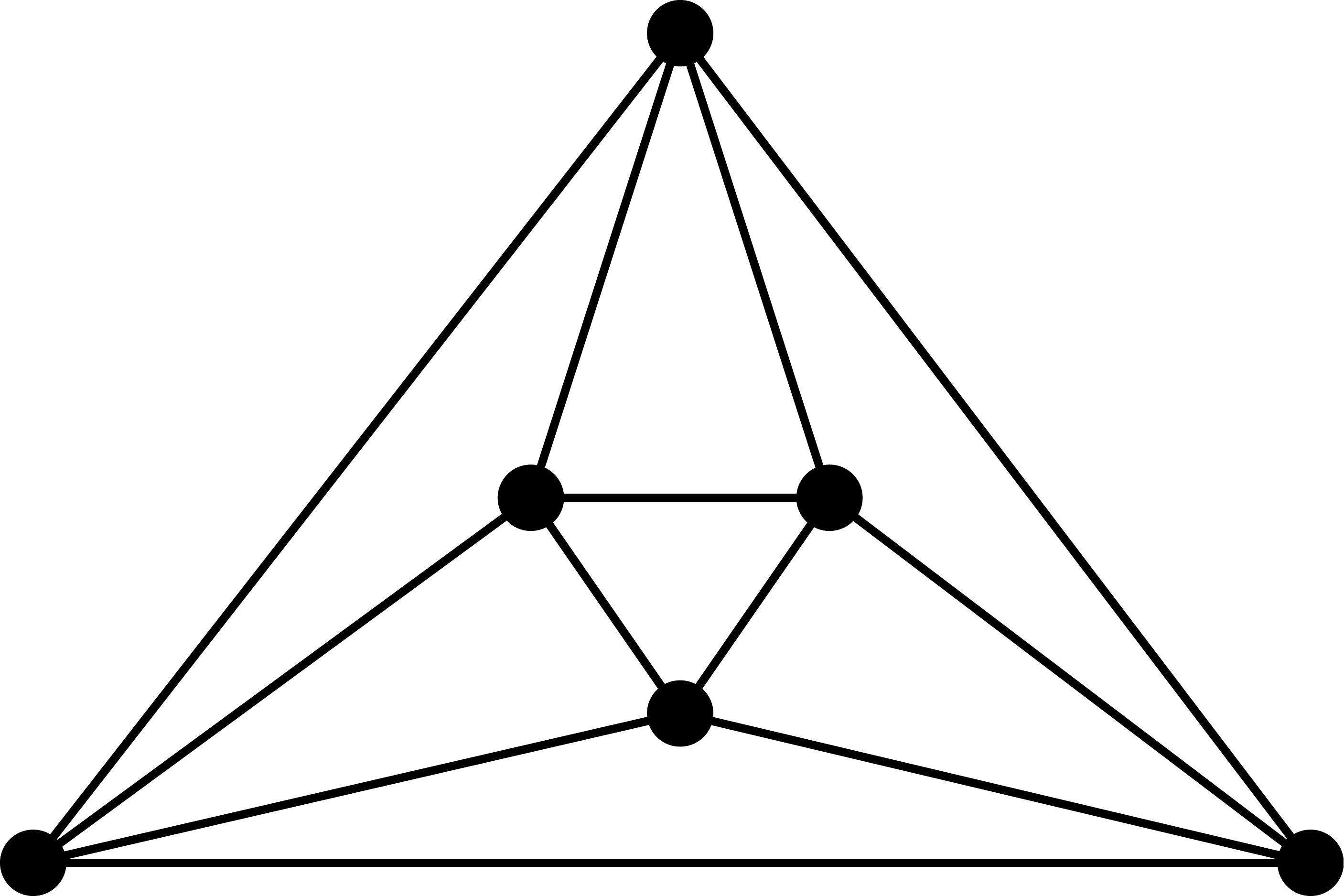}
\hspace{0.5in}
\includegraphics[height=1.3in]{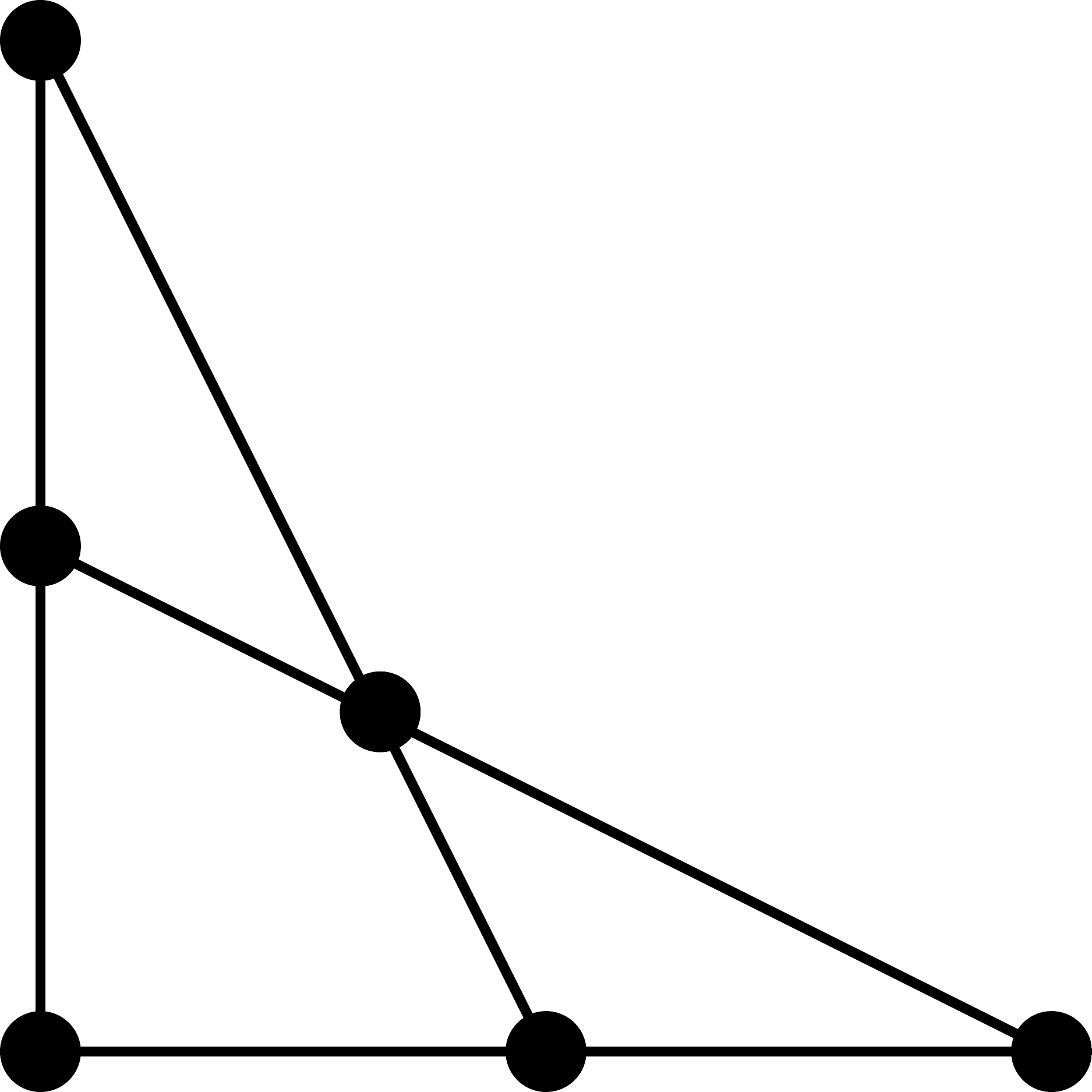}
\end{center}
\caption{The graph from Example~\ref{e:planarnotplanar} (left) and the 3-segment hypergraph from Example~\ref{e:3segment} (right).}
\label{f:3segment}
\end{figure}

\section{Coloring segment hypergraphs}
\label{sec:rsegment}

We now consider the following family of hypergraphs introduced in \cite{rsegment}.  Note that a graph $G$ is a 2-segment hypergraph if and only if $G$ has a sequential embedding in $\RR^2$.  

\begin{defn}[{\cite[Definition~1.1]{rsegment}}]\label{d:rsegment}
A hypergraph $H$ with vertex set in $\ZZ^2$ is called an \emph{$r$-segment hypergraph} if (i) every edge of $H$ consists of $r$ consecutive integer points and (ii) every line in $\RR^2$ contains at most one edge of $H$.  
\end{defn}

\begin{remark}\label{r:2segmentsameline}
The special case of Definition~\ref{d:rsegment} when $r = 2$ (that is, when $H$ is a graph)  has an additional hypothesis that Definition~\ref{d:sequential} does not, namely that each line in $\RR^2$ contain at most one edge.  It can be easily shown that the constructions in Theorem~\ref{t:2segmentrealizable} and~\ref{t:planarseq} also satisfy this property.  
\end{remark}

We say a vertex coloring of a hypergraph $H$ is a \emph{strong coloring} if no two vertices sharing an edge are the same color.  Likewise, we say a vertex coloring of $H$ is a \emph{weak coloring} if no edge is monochromatic, that is, if no edge consists entirely of vertices of a single color.  We denote the \emph{strong} and \emph{weak chromatic numbers} of $H$ by $\chi_s(H)$ and $\chi_w(H)$, respectively (that is, the smallest number of colors needed to strongly (resp.\ weakly) color the vertices of $H$).  

Given an $r$-segment hypergraph $H$, \cite[Theorem~2.2]{rsegment} states that $\chi_w(H) \le 4$ if $r = 2$, $\chi_w(H) \le 3$ if $r = 3$, and $\chi_w(H) = 2$ if $r \ge 4$.  Additionally, \cite[Examples~2.3 and~2.4]{rsegment} demonstrate that each of these bounds is sharp by exhibiting 2-segment and 3-segment hypergraphs with weak chromatic numbers 4 and 3, respectively.  

In this section, we bound the strong chromatic number of $r$-segment hypergraphs in terms of $r$ (Theorem~\ref{t:strongbound}) and demonstrate that our bound is sharp (Proposition~\ref{p:strongsharp}).  Lastly, we demonstrate in Example~\ref{e:3segment} that the strong and weak chromatic numbers of an $r$-uniform hypergraph $H$ are not sufficent to determine if $H$ is an $r$-segment hypergraph for $r \ge 3$.  

We begin by recalling a particular family of hypergraphs from~\cite{rsegment}.  For each $k \ge 2$, let $\ZZ_k$ denote the additive group $\ZZ/k\ZZ = \{\overline{0}, \overline{1}, \ldots, \overline{k-1}\}$.  Let $Z_k$ denote the $k$-uniform hypergraph with vertex set $V(Z_k) = \ZZ_k^2$ and edge set
\[
E(Z_k) = \{\ol u + \ZZ_k \ol v : \ol u, \ol v \in \ZZ_k^2 \text{ and } |\ol u + \ZZ_k \ol v| = k\}
\]
consisting of all lines in $\ZZ_k^2$ with exactly $k$ points.  

\begin{thm}\label{t:strongbound}
If $H$ is an $r$-segment hypergraph with $r \ge 2$, then $\chi_s(H) \le r^2$.  
\end{thm}

\begin{proof}
By \cite[Proposition~2.1]{rsegment}, the image of any edge of $H$ under the projection
$$\begin{array}{r@{}c@{}l}
\ZZ^2 &{}\longrightarrow{}& \ZZ_r^2 \\
(v_1 ,v_2) &{}\longmapsto{}& (\ol v_1, \ol v_2)
\end{array}$$
is an edge in $Z_r$.  In particular, any $r$ sequential vertices on an edge of $H$ map to distinct vertices of $Z_r$.  As such, any strong coloring of $Z_r$ induces a strong coloring of~$H$, and thus $\chi_s(H) \le \chi_s(Z_r) \le |\ZZ_r^2| = r^2$. 
%
\end{proof}

\begin{remark}\label{r:strongweakr2}
If $r = 2$, we have $\chi_w(H) = \chi_s(H)$ for every $2$-segment (hyper)graph $H$.  In this case, the bounds in Theorem~\ref{t:strongbound} and \cite[Theorem~2.2]{rsegment} coincide.  
\end{remark}

We now show the bound in Theorem~\ref{t:strongbound} is sharp by exhibiting, for each $r \ge 2$, an $r$-segment hypergraph with strong chromatic number $r^2$ (see Figure~\ref{f:strongsharp} for a depiction of the resulting hypergraph for $r = 3$).  

\begin{figure}
\begin{center}
\includegraphics[height=2.0in]{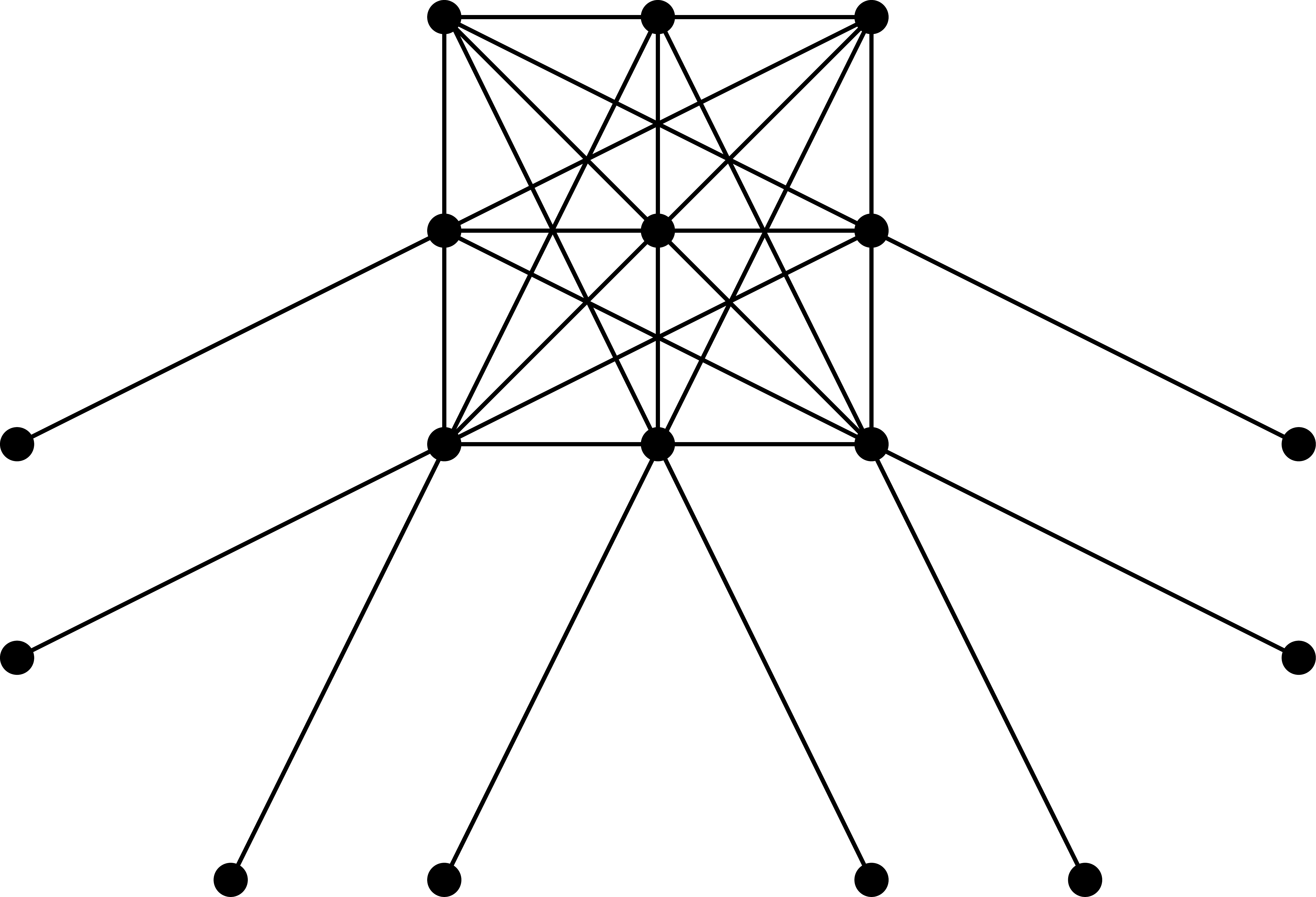}
\end{center}
\caption{The $3$-uniform hypergraph from Proposition~\ref{p:strongsharp}.}
\label{f:strongsharp}
\end{figure}

\begin{prop}\label{p:strongsharp}
Let $r \ge 2$. There is an $r$-segment hypergraph $H$ with $\chi_s(H) = r^2$.
\end{prop}

\begin{proof}
Let $B = \{0, 1, \ldots, r - 1\}^2 \subset \ZZ^2$, an $r \times r$ grid of integer points.  Consider all lines in $\RR^2$ passing through at least two of the points in $B$.  Each such line $L$ contains at most $r$ points in $B$, since there are at most $r$ consecutive points in $B$ in any given direction.  Let $H$ denote the hypergraph with one edge $E_L$ for each line $L$ consisting of $r$ sequential integer points in $L$ satisfying $E_L \supset L \cap B$.  By~construction, $H$ is an $r$-segment hypergraph with the property that any two vertices in $B$ share an edge in $H$.  As such, under any strong coloring of $H$, the vertices in $B$ must have distinct colors.  This implies $\chi_s(H) \ge r^2$, so by Theorem~\ref{t:strongbound}, $\chi_s(H) = r^2$.
\end{proof}

We conclude with an example demonstrating that $r$-segment hypergraphs for $r > 2$ cannot be characterized in terms of their (strong or weak) chromatic numbers.  

\begin{example}\label{e:3segment}
Let $H$ be the hypergraph with vertex set $V(H) = \{1,2,3,4,5,6\}$ and edge set $E(H) = \{\{1,2,3\},\{1,4,5\},\{2,5,6\},\{3,4,6\}\}$ depicted in Figure~\ref{f:3segment}. It is easy to check that $\chi_s(H) = 3$, but no $3$-segment hypergraph is isomorphic to $H$.  Indeed, every 2 edges of $H$ intersect, and no 3 edges of $H$ have a common intersection, so 
a sequential embedding in the plane is impossible by \cite[Lemma~4.5]{rsegment}.  
\end{example}


\end{document}